\documentclass[12pt,letterpaper,reqno]{amsart}

\usepackage{amsthm}
\usepackage{mathrsfs}
\usepackage{amsfonts,amssymb,amsmath}
\input amssym.def 
\input amssym.tex

\usepackage[dvips]{graphicx}
\usepackage{hyperref}
\usepackage{thmtools}

\declaretheoremstyle[
bodyfont=\normalfont,
]{remstyle}

\newtheorem{defi}{\bf Definition}

\newtheorem{lemma}{\bf Lemma}

\newtheorem{theorem}{\bf Theorem}

\newtheorem{remark}{\bf Remark}

\newcommand\be{\begin{eqnarray*}}
\newcommand\ee{\end{eqnarray*}}
\newcommand\beq{\begin{equation}}
\newcommand\eeq{\end{equation}}
\newcommand\eps{\epsilon}

\newcommand{\Cal}[1]{\mathcal{#1}}

\newcommand\ben{\begin{eqnarray}}
\newcommand\een{\end{eqnarray}}

\begin{document}


\date{\today}

\title[Discrete spheres and arithmetic progressions in product sets] 
{Discrete spheres and arithmetic progressions in product sets}

\author{Dmitrii Zhelezov}
\thanks{Department of Mathematical Sciences, 
Chalmers University Of Technology and University of Gothenburg} 
\address{Department of Mathematical Sciences, 
Chalmers University Of Technology and University of Gothenburg,
41296 Gothenburg, Sweden} \email{zhelezov@chalmers.se}


\date{\today}

\begin{abstract}
	We prove that if $B$ is a set of $N$ positive integers such that $B\cdot B$ contains an arithmetic progression of length $M$, then for some absolute $C > 0$,
	$$
	 \pi(M) + C \frac {M^{2/3}}{\log^2 M} \leq N,
	$$  
	where $\pi$ is the prime counting function.
	This improves on previously known bounds of the form $N = \Omega(\pi(M))$ and gives a bound which is sharp up to the second order term, as P\'ach and S\'andor gave an example for which 
	$$
	N < \pi(M)+ O\left(\frac {M^{2/3}}{\log^2 M} \right).
	$$ 
	
   The main new tool is a reduction of the original problem to the question of approximate additive decomposition of the $3$-sphere in $\mathbb{F}_3^n$ which is the set of 0-1 vectors with exactly three non-zero coordinates. Namely, we prove that such a set cannot have an additive basis of order two of size less than $c n^2$ with absolute constant $c > 0$. 
\end{abstract}

\maketitle

\section{Introduction}
	One of the main objectives of additive combinatorics is to show that a set cannot be simultaneously additively and multiplicatively structured. In particular, if one defines the sumset 
	$$
	A+A = \{a + a' : a,a' \in A \}
	$$   
	and the product set
	$$
	A \cdot A = \{aa' : a,a' \in A \}	
	$$
	of a set $A$, the sum-product conjecture of Erd\H{o}s and Szemer\'edi claims that the size of at least one of these sets is substantially larger than that of the original set. Originally, the conjecture was formulated only for integer sets, but later it became clear that conceivably it should hold in a much more general setting, so now it is usually stated for $A \subset \mathbb{C}$. More precisely,
\beq \label{eq:ErdosSz}
		\max \{|A \cdot A|, |A+A| \} \geq C_\eps |A|^{2 - \eps}, 	
	\eeq
	where $\eps > 0$ can be taken arbitrarily small and $C_\eps$ is some constant which depends only on $\eps$. There are numerous results towards this conjecture in different flavours but it is still far from resolution. We refer the interested reader to \cite{TaoSurvey} for an overview of this topic\footnote{Further progress has been made since the paper \cite{TaoSurvey} was published, but it still gives a great overview of the sum-product phenomenon. See also the standard reference \cite{TaoVu}.}.
	
	On the other hand, very little is known about sets which are themselves of the form $B \cdot B$ or $B+B$, or subsets thereof. Clearly, such sets must be as well multiplicatively resp. additively structured in a certain sense, but perhaps in a more subtle way. For example, it is an open problem to algorithmically decide whether a given set $A$ is of the form $B+B$, see Problem 4.11 in \cite{CrootProblems}. 
	
	By analogy to the classical sum-product phenomenon, one might expect that $B \cdot B$ (resp. $B+B$) cannot contain large additively (resp. multiplicatively) structured pieces, with some freedom left in defining what the actual structure is. Some results in this direction have been obtained by Roche-Newton and the author \cite{RNZ}, \cite{Z1}, where the structure is defined in terms of the additive (resp. multiplicative) energy.  In the extreme case of a product set $B \cdot B$ containing an arithmetic progression (AP), the author showed in \cite{Z2} that the maximal length of the arithmetic progression must be  $O(|B|\log |B|)$, at least in the integer case, together with an example exhibiting a progression of length $\Omega(|B|\log |B|)$.

	The goal of the present note is to further refine these bounds in the integer case by recasting the original problem to $\mathbb{F}_3^n$. The linear space structure of $\mathbb{F}_3^n$ allows one to obtain more control over the structure of the sets in question and eventually significantly refine the previous bound, which was based mainly on number theoretic arguments. The heart of the proof is Lemma 4 below about minimal additive decomposition of discrete spheres in $\mathbb{F}_3^n$ which might be of independent interest. 
	
	One can think of the problem considered in the present paper as the problem of finding a uniform bound for the minimal size of a mutiplicative basis	\footnote{The reader might be familiar with \emph{additive bases} which are defined analogously but with multiplication replaced with addition. }  	(of order two)
 for an integer arithmetic progression of size $N$. 
 
 Indeed, recall that a set $B$ is a \emph{multiplicative basis} for a set of numbers $A$ if $A \subset B \cdot B$. For a given set $A$, define 
 $$
 MB_\times(A) := \min \{|B| : A \subset B \cdot B \}
 $$ and 
$$
	MBP_\times (M) := \min \{MB_\times(P) : P \text{ is an AP of length } M \}.
$$ 
The main result of the present paper can be reformulated as follows. 
\begin{theorem} \label{thm:main} 
There is an absolute $C > 0$ such that 
$$
	 \pi(M) + C \frac {M^{2/3}}{\log^2 M} \leq MBP_\times (M),
$$
where $\pi$ is the prime counting function. 
\end{theorem}

Clearly,  in order to bound $MBP_\times (M)$ from above it suffices to construct a thin multiplicative basis for the integer interval $[M]$. This problem was recently considered by Pach and S\'andor \cite{PachSandor} (see also references therein), who constructed an essentially best possible multiplicative basis. In particular, they proved the following.

\begin{theorem}[Pach and S\'andor, \cite{PachSandor}] \label{thm:PachSandor}
$$
MB_\times([M]) \leq \pi(M) + 301 \frac {M^{2/3}}{\log^2 M}.
$$
\end{theorem}
Theorem \ref{thm:PachSandor} shows that the bound of Theorem \ref{thm:main} is essentially best possible up to the constant in the second order term. It is conceivable that in fact 
$MBP_\times (M) = MB_\times([M])$ (and  the proof of Theorem \ref{thm:main} gives a rather strong support for this claim), but we do not address this issue in the present paper. It is also an open question to establish the best possible constant in the second order term of $MB_\times([M])$ and $MBP_\times (M)$. 
 	
\section{Notation}	
	Throughout the paper we use the standard notation $X \ll Y$, $Y \gg X$, $X = O(Y)$, $Y = \Omega(X)$. All of these expressions mean that there is a constant $C$, independent of $X$ and $Y$, such that $|X| \leq CY$. 
		
	We write $(a, b)$ for the greatest common divisors of two integers $a$ and $b$. For a prime $p$, we write $v_p(x)$ for the maximal power $f$ such that $p^f | x$. We also write 
	$p^f\| x$ if $f = v_p(x)$, that is $p^f | x$ but $p^{f+1} \nmid x$.
	
	For a subset $X$ of a linear space $V$ we write $\mathrm{Span}(X)$ for the linear span of the elements of $X$.

\section{Main result}	
		
Throughout the paper we will assume that an arithmetic progression $A$ of size $M$ is contained in the product set $B \cdot B$ of an integer set $B$ of size $N$. Our goal is to bound 
$MBP_\times (M)$, that is, to give a uniform lower bound on $N$ in terms of $M$.

Obviously, a lower (upper) bound on $N$ in terms of $M$ is equivalent to an upper (lower) bound on $M$ in terms of $N$. We will prove Theorem \ref{thm:main} by bounding $M$ in terms of $N$ rather then what is stated, as it is slightly more convenient for calculations.

Both Lemma \ref{lm:reduction} and \ref{lm:small_ap} below are essentially borrowed from \cite{Z2} and presented here in order to make the exposition self-contained.

\begin{lemma}[Reduction of sets] \label{lm:reduction} Suppose that there exist a set $B$ of $N$ positive integers, such that there is an $M$-term arithmetic progression $A$ such that $A \subset B \cdot B$. Then there exist sets $A', B'$ such that 
\begin{enumerate}
	\item $|A'| = |A| = M$,
	\item $|B'| = |B| = N$,
	\item $A' =\{ a+md:\ 1\leq m\leq M\} \subset B'\cdot B'$,
	\item if $g=(a,d)$ with $a=gu$ and $d=gv$ then $(v,g)=1$
\end{enumerate}

 For these sets we have the following property:

 If $\Cal{M}\subset \{ m:\  1\leq m\leq M\}$ is such that for each $m\in \Cal{M}$ there exists a prime $p_m$ for which $p_m|u+vm$ and $p_m\nmid u+vm'$ for all $m'\in \Cal{M}\setminus \{m\}$, then 
 $|B'| \geq |\Cal{M}|$.

\end{lemma}

\begin{proof}
Let $(A,B)$ be a pair of sets with the product of the elements of $B$ minimal among all pairs of sets satisfying (1)--(3) (note that there could be many such pairs and there is at least one such a pair). If $(v,g)\ne 1$ then there exists a prime $p$ and exponents $e>f\geq 1$ with $p^e\| d$ and $p^f\| a$. Therefore $p^f\| x$ for every $x \in A$, which means that if  $x = b_1b_2$ then 
$$
\text{either} \ 0<v_p(b_1), v_p(b_2)<f; \qquad \text{or} \ v_p(b_1)=0 \text{ and } v_p(b_2)=f,
$$
(where $v_p(b)$ is the exact power of $p$ dividing $b$). 

If $f = 1$ then let $A' = \{ x/p:\ x\in A\}$ and 
$$
B' := \{ b\in B:\ v_p(b)=0\}\  \cup\  \{ b/p: \ b\in B \ \& \ 0<v_p(b) = f\}.
$$
Otherwise, let $A'=\{ x/p^2:\ x\in A\}$ and 
$$
B':= \{ b\in B:\ v_p(b)=0\}\  \cup\  \{ b/p: \ b\in B \ \& \ 0<v_p(b)<f\} \  \cup \  \{ b/p^2: \ b\in B \ \& \  v_p(b)\geq f\}.
$$
One easily verifies that $A'\subset B'\cdot B'$ with $|A'|=M, \ |B'|=N$ and the product of the elements of $B'$ is much smaller than the product of the elements of $B$. This  contradicts our assumption of minimality.

Now we prove the second part of the lemma. Let $\mathcal{P}$ be the set of primes $p_m$ with $m \in \Cal{M}$, so $|\Cal{P}| = |\Cal{M}| := M_p$. Define a map $\rho: \mathbb{Z} \to (\mathbb{F}_q)^{M_p}$, where $q$ is some prime number defined in due course, as follows. If $x = D\prod_{m \in \Cal{M}}p^{e_m}_m$, then $\rho(x)_i = e_i \mod q$. In words, the $i$th coordinate of $\rho(x)$ is $v_{p_i}(x)$ modulo $q$. Clearly, $\rho(A') \subset \rho(B') + \rho(B')$ and $|B'| \geq |\rho(B')|$ . 

Now let $\rho_g = \rho(g)$ and define $\check{B} = \rho(B') - \rho_g / 2$, where the division is coordinate-wise. It is easy to see that $\rho(u + vm) \in \check{B} + \check{B}$ for any $m \in \Cal{M}$. Moreover, by the hypothesis of the lemma we can choose $q$ such that $\rho(u + vm)$ has only one non-zero coordinate, namely $m$. But then obviously the linear span $\mathrm{Span}(\check{B})$ in $(\mathbb{F}_q)^{M_p}$ must be the whole space, so $|\check{B}| \geq M_p = |\Cal{M}|$. The desired property now follows.

\end{proof}

From now on we will assume that the pair of sets $(A, B)$ in question satisfies the properties (1)--(4) of Lemma \ref{lm:reduction}. 

\begin{lemma}[Further reduction to $v=1$] \label{lm:small_ap}
	Using the notation of Lemma \ref{lm:reduction}, if $|B| \leq \pi(M) + M^{2/3}$, one can assume that $v = 1$, $u \ll M^{2/3}\log M$.
\end{lemma}

\begin{proof}[Proof of upper bound on $M$.]  We begin by using the lemma. If $p$ is a prime $\geq M$ then it can divide at most one $u+mv$ (since $(u,v)=1$). Let 
$$
\Cal{M} := \{ m:\ 1\leq m\leq M \ \& \ \exists \ p\geq M:\ p| u+mv\} .
$$

Now if $1\leq m\leq M$ and $m\not \in \Cal{M}$ then all prime factors of $u+mv$ are $<M$. 
For each prime $p<M$ select $m_p$ for which $v_p(u+m_pv)$ is maximized and then 
$ \Cal E := \cup_{p<M} \{ m_p\} $.  We claim that
$$  
 \prod_{\substack{1\leq m \leq M \\ m\not\in  \Cal{M} \cup  \Cal{E}}}(u+mv) \ \ \text{divides}  \ \ (M-1)!
 $$
Indeed, the left hand side has only prime divisors less than $M$ and for such a prime $p < M$ the maximal power of $p$ which divides the left hand side is at most
$$
	\sum_{k=1}^{\infty} \big[ \frac{M-1}{p^k} \big].
$$
This is because if $k \leq v_p(u+m_pv)$ the only elements of $\{u + mv \}$ divisible by $p^k$ are of the form $ u + (m_p + jp^k)v, j \in \mathbb{Z}$. But the same total power of $p$ also divides $(M-1)!$.

Let $L:=|\Cal{M} \cup  \Cal{E}| \leq |\Cal{M} |+| \Cal{E}|< |B|+\pi(M-1)$. Hence 
\beq \label{eq:bin_estimate}
  \prod_{m=1}^{M-L}   (u+mv) \leq (M-1)!  
\eeq
If we assume that $|B|\leq \{ 1+o(1)\} M/\log M$ then $L\leq \{ 2+o(1)\} M/\log M$ by the prime number theorem.
If  $v\geq 9$ then (\ref{eq:bin_estimate}) yields\footnote{The absolute constants are probably not optimal, but it doesn't affect the later steps of the proof.}, using Stirling's formula,
\be
 (9/e^{2+o(1)})^M (M/e)^M  &\ll& (9(M-L)/e)^{M-L} \\
 										  &\ll&(M-L)! v^{M-L} \leq (M-1)!  \ll (M/e)^M ,
\ee
which gives a contradiction for $M$ sufficiently large.

If $u\geq 7M$ then (\ref{eq:bin_estimate}) yields, using Stirling's formula,
\be
 (8^8/e^{2+o(1)})^M (M/e)^{8M} &\ll& ((8M-L)/e)^{8M-L}  \\
 											  &\ll& (8M-L)! \leq (7M)!(M-1)! \ll (7M/e)^{7M}(M/e)^{M}, 
\ee
 which gives a contradiction for $M$ sufficiently large.

   Hence we may assume that $v\leq 9$ and $u<7M$, so that each $u+mv\leq 16M$. Any element of the arithmetic progression $\{ u + mv \}$ is divisible by at most one prime $p>4\sqrt{M}$ and then to the first power. For each prime $4\sqrt{M}<p\leq M$, select some $m_p$ for which $p|u+m_pv$. We let $\Cal M = \ \{ m_p: \ 4\sqrt{M}<p\leq M\}$ and apply Lemma \ref{lm:reduction}, so that $|B|\geq \pi(M)-\pi(4\sqrt{M})$.

We can do better than this. We can take 
$$
\Cal{M} = \ \{ m_p: \ 4\sqrt{M}<p\leq M\}\cup \{ p>M:\ p=u+mv\  \text{ for some } \ 1\leq m\leq M\}
$$
so that $|\Cal{M}|=\pi(M)-\pi(4\sqrt{M})+\pi(u+Mv;v,u)-\pi(\max\{ M,u\};v,u)$. This latter quantity $\pi(u+Mv;v,u)-\pi(\max\{ M,u\};v,u)$ is $>\epsilon M/\log M$ unless 
$v=1$ and $u\ll \epsilon M$.  In fact to get as small as the set already constructed we must have $u\ll M^{2/3}\log M$ and $v = 1$, which we will assume from now on. 
\end{proof}

We proceed with the following simple lemma which says that the regime $v = 1, u \ll M$ is essentially the same as $v =1, u = 0$.

\begin{lemma} \label{lm:power_primes}
	Let $I$ be an integer interval  of size $M$ and step one contained in the interval $[2M]$. Then for any $x \in [M]$ there is $k \geq 0$ such that $2^kx \in I$.
\end{lemma} 
\begin{proof}
	Let $x \in [M]$ and $I =  [a+1, a + M]$ for some $0 \leq a \leq M$. If $x \in I$ we simply take $k = 0$. Otherwise, let $k$ be the smallest integer such that $2^{k+1}x > a + M$. Then $2^kx \leq a + M$ by construction. On the other hand $2^kx > a$ since $a \leq M$, so $2^kx \in I$.
\end{proof}

%Hence we may assume that $v\leq 8$ and $u<7M$, so that each $u+mv\leq 15M$. But then we can simply take $\Cal{M} = \{ m : u + mv \text{ is prime} \}$, and by the Dirichlet theorem on primes in %arithmetic progressions, $|\Cal{M}| \geq \pi(M)$ for $M$ sufficiently large. Then by Lemma \ref{lm:reduction}, $|B| \geq \pi(M)$.

%Any element of the arithmetic progression is divisible by at most one prime $p>\sqrt{15M}$ and then to the first power. For each prime $\sqrt{15M}<p\leq M$, select some $m_p$ for which $p|u+m_pv$. %We let $\Cal M = \ \{ m_p: \ \sqrt{15M}<p\leq M\}$ and apply (ii), so that $|B|\geq \pi(M)-\pi(\sqrt{15M})$.

%We can do better than this. We can take 
%$$
%\Cal M = \ \{ m_p: \ \sqrt{27M}<p\leq M\}\cup \{ p>M:\ p=u+mv\  \text{for some} \ 1\leq m\leq M\}
%$$
%so that $|\Cal M|=\pi(M)-\pi(\sqrt{27M})+\pi(u+Mv;v,u)-\pi(\max\{ M,u\};v,u)$. This latter quantity is $>\epsilon M/\log M$ unless 
%$v=1$ and $u\ll \epsilon M$. In fact to get as small as the set already constructed we must have $u\ll M^{2/3}\log M$. However this means that the elements of the arithmetic progression are $\leq %M+O(M^{2/3}\log M)$ and so
%we can take our primes $p>\sqrt{M}+O(M^{1/6}\log M)$.  Hence we would have 
%$|\Cal M|=\pi(M+u)-\pi(\sqrt{M})+O(M^{1/6})$, using the Brun-Titchmarsh theorem.

Recall that using Lemma \ref{lm:reduction} and Lemma \ref{lm:small_ap} we have reduced the original pair $(A, B)$ to the case  $A = \{ g(u + m) \}, 1 \leq m \leq M$ with $A \subset B \cdot B$.  As $u \ll M^{2/3} \log M$, it follows that $u \leq M$, for $M$ large enough (i.e. larger than a fixed constant which in turn only affects the constant in the final bound of Theorem \ref{thm:main}). Similarly, in what follows we assume that $M$ is `sufficiently large' for all the asymptotic estimates to hold.

Let $\Cal{P}_1$ be the set of primes in the interval  $[M^{1/3} + 1, M]$ and $\Cal{P}_2$ be the set of primes in the interval $[3, M^{1/3}]$. By Lemma \ref{lm:power_primes}, we can pick  for each integer $x \in  [M]$ some $m_x$ such that $u + m_x = 2^{k(x)}x$. We define 
 $$
 \Cal{M}_1 = \{1 \leq m_p \leq M,  \,\, p \in \Cal{P}_1 \}
 $$ and
$$
\Cal{M}_2 = \{1 \leq m_x \leq M,  \,\, x = p_ip_jp_k \},
$$ 
where $p_i, p_j, p_k$ are distinct primes in $\Cal{P}_2$.  

 By the Prime Number Theorem, 
$$
|\Cal{M}_1| \geq \pi(M) - \pi(M^{1/3})
$$ 
and 
$$
|\Cal{M}_2| \gg \pi(M^{1/3})^3 \gg \frac{M}{\log^3 M}.
$$ 
Obviously, $\Cal{M}_1$ and $\Cal{M}_2$ as well as $\Cal{P}_1$ and $\Cal{P}_2$ are disjoint.  
We will now use a map similar to the one used in the proof of Lemma \ref{lm:reduction}, namely $\rho: \mathbb{Z} \to \mathbb{F}^{|\Cal{P}_1| + |\Cal{P}_2|}_3$. The map requires a bit of notation and is defined as follows. 

We assign a coordinate index in $\mathbb{F}^{|\Cal{P}_1| + |\Cal{P}_2|}_3$ to each element of $\Cal{P}_1 \cup \Cal{P}_2$ by mapping $\Cal{P}_1$ to the first $|\Cal{P}_1|$ coordinate indices and $\Cal{P}_2$ to the last $|\Cal{P}_2|$ coordinate indices. We  thus have defined a one-to-one correspondence between the primes in $\Cal{P}_1 \cup \Cal{P}_2$ and the coordinates of the linear space. Further, for each prime $p$ in  $\Cal{P}_1 \cup \Cal{P}_2$ the correspondence induces the canonical coordinate projection $\pi_p:  \mathbb{F}^{|\Cal{P}_1| + |\Cal{P}_2|}_3 \to \mathbb{F}_3$ which maps an element $x \in  \mathbb{F}^{|\Cal{P}_1| + |\Cal{P}_2|}_3$ to $x_p \in \mathbb{F}_3$, which is the coordinate of $x$ corresponding to the prime $p$. We will use the subscript notation for the action of coordinate projections from now on.

We now define the map $\rho: \mathbb{Z} \to \mathbb{F}^{|\Cal{P}_1| + |\Cal{P}_2|}_3$ coordinate-wise. For a prime $p_i \in \Cal{P}_1 \cup \Cal{P}_2$, we define  $\rho(\cdot)_{p_i} := v_{p_i}(\cdot) \mod 3$,  in other words, each coordinate is defined as the maximal power of the correspondent prime reduced modulo three.

Consider  the set $B' = \rho(B) - \rho(g)/2$. Then, $\rho(u + m) \subset B' + B'$, in particular for 
$m \in \Cal{M}_1 \cup \Cal{M}_2$.  Let $A' = \rho(u + m)$ with $m$ restricted to $\Cal{M}_1 \cup \Cal{M}_2$.

Now we are going to examine closely what $A'$ looks like. It will be convenient to introduce the following definition. 
\begin{defi}
	The \emph{Hamming $k$-sphere} (or simply \emph{$k$-sphere}) in  $\mathbb{F}^n_q$ is the set $S_k \subset \mathbb{F}^n_q$ of $\{0, 1\}$ vectors with exactly $k$ non-zero coordinates. 
\end{defi}

%The reader might note that for $q = 2$ the $k$-sphere is just a Hamming sphere of radius $k$ centered at zero.

Our construction of $\Cal{M}_1$ and $\Cal{M}_2$ reveals that $A'$ restricted to the first $|\Cal{P}_1|$ coordinates is a $1$-sphere, and restricted to the last $|\Cal{P}_2|$ coordinates is a $3$-sphere.   
\\
\\
We now state a lemma about $3$-spheres, which might be of independent interest, but defer the proof till the end of the current discussion. 

\begin{lemma} \label{lm:additive_sphere}
  Let $S_3$ be the $3$-sphere in $\mathbb{F}^n_3$. If $S_3 \subset B + B $ then $|B| = \Omega(n^2)$.
\end{lemma}

%We strongly believe that  Lemma \ref{lm:additive_sphere} can be improved to $|B| = \Omega(n^2)$ which would in turn give an almost sharp upper bound $\pi(M) + M^{2/3 - o(1)}$ in Theorem 1, but %we were unable to prove a stronger version.

Using Lemma \ref{lm:additive_sphere} as a black box, we now finish the proof of Theorem \ref{thm:main}.

\begin{proof}[Proof of Theorem \ref{thm:main}]
	Let $A' \subset B' + B'$ be the sets defined in the preceding discussion. 	Let $H$ be a graph with $B'$ as the vertex set and for each $a \in A'$ a single edge is placed joining some pair $(b, b')$ with $b + b' = a$. In fact, the edges are of two types: the ones corresponding to $m \in \Cal{M}_1$ and to $m \in \Cal{M}_2$.  Consider the subgraph $H'$ spanned by the edges of the first type, so $E(H')$ restricted to the first $|\Cal{P}_1|$ coordinates is a $1$-sphere and the projection to the last $|\Cal{P}_2|$ coordinates is zero. 
	
	Let $H_1, \ldots, H_K$ be the connected components of $H'$. It follows that $V(H_i)|_{\Cal{P}_2}$ (i.e. projection to the last $|\Cal{P}_2|$ coordinates) can attain only two values on each component and they must sum to zero. Denote these vectors by $\vec{v}_i$ and $-\vec{v}_i$. Moreover, if $H_i$ contains an odd cycle,  $\vec{v}_i = \vec{0}$, since chasing such a cycle implies $\vec{v}_i = -\vec{v}_i$. On the other hand, $H_i$ cannot contain even cycles at all, since it would imply a linear dependence between edges, which is absurd as they form a basis for the subspace spanned by the first  $|\Cal{P}_1|$ coordinates.  
	
	Let $T$ be the number of trees among $H_i$. Summing up, we have the following observations. First, 
	\beq \label{eq:treebound}
	 |V(H')|_{\Cal{P}_2}| \leq 2T + 1,
	\eeq
	since the vertices of any component which is not a tree have trivial projection to $\Cal{P}_2$-coordinates and for each tree the vertices are projected onto at most two vectors. Second,  $|V(H_i)| \geq |E(H_i)|$ since $V(H_i)$ must span the subspace generated by $E(H_i)$ which has dimension $|E(H_i)|$. Of course, in the case of a tree we have $|V(H_i)| = |E(H_i)| + 1$. Thus, we have 
	\beq \label{eq:tree_estimate}
	 \sum_i |E(H_i)| + T \leq  |V(H')|.
	\eeq
	
	Let now $B_S = (B' \setminus V(H'))\big|_{\Cal{P}_2} \cup V(H')\big|_{\Cal{P}_2}$.  It is trivial that $B_S|_{\Cal{P}_2} = B'|_{\Cal{P}_2}$. Thus, $B_S|_{\Cal{P}_2} + B_S|_{\Cal{P}_2}$ contains a 3-sphere (which corresponds to the elements in $\Cal{M}_2$), so we can apply Lemma \ref{lm:additive_sphere} and obtain
	\beq \label{eq:s3estimate}
	 |(B' \setminus V(H'))\big|_{\Cal{P}_2}| + |V(H')\big|_{\Cal{P}_2}| \geq |B_S| \geq |B_S|_{\Cal{P}_2}| \gg  |\Cal{P}_2|^{2} \geq \frac{M^{2/3}}{\log^2 M}.
	\eeq 

We have by (\ref{eq:treebound}), (\ref{eq:tree_estimate}), (\ref{eq:s3estimate}) and the fact that $|E(H')| = |\Cal{M}_1|$
\be
	|B| \geq |B'| &=&  |V(H')| + |B' \setminus V(H')| \geq \\ 
		 &\geq& |E(H')| + T +  |(B' \setminus V(H'))\big|_{\Cal{P}_2}|	\geq \\
		 &\geq& |\Cal{M}_1| + \frac{1}{2}( |V(H')|_{\Cal{P}_2}| +  |(B' \setminus V(H'))\big|_{\Cal{P}_2}|) - 1 \geq \\
		 &\geq& \pi(M) - \pi(M^{1/3}) + \Omega \left(\frac{M^{2/3}}{\log^2 M} \right) = \pi(M) + \Omega \left(\frac{M^{2/3}}{\log^2 M} \right)
\ee

This finishes the proof.

%	which means that at least (moving the implicit constant in $\gg$ into the $o(1)$ term)
%	$$
%	 |\Cal{M}_2|^{2/3 -o(1)}  \leq |V(H')\big|_{\Cal{P}_2}| \leq 2T + 1
%	$$ 
%	or 
%	$$
%	 |\Cal{M}_2|^{2/3 -o(1)} \leq |(B' \setminus V(H'))_{\Cal{P}_2}| \leq |B' \setminus V(H')|.
%	$$
%	In any case we have (suppressing constants into the $o(1)$ term)
%	$$
%	|\Cal{M}_1| + |\Cal{M}_2|^{2/3 -o(1)} \leq |V(H')| + |B' \setminus V(H')| = |B'|,
%	$$ so unwinding the definitions and using the trivial $|B'| \leq |B|$, we obtain (again hiding lower order terms in $o(1)$)
%	$$
%	\pi(M) + M^{2/3 - o(1)} \leq \pi(M) - \pi(M^{1/3}) + M^{2/3 - o(1)} \leq |\Cal{M}_1| + |\Cal{M}_2|^{2/3 -o(1)} \leq |B|.
%	$$
		
\end{proof}

It remains to prove Lemma \ref{lm:additive_sphere}.

\begin{proof}[Proof of Lemma \ref{lm:additive_sphere}]
	Let us examine what a difference between two elements $a, a' \in S_3$ can look like. To this end, let us fix $k \in \mathbb{F}_3^n$ and look at the number of solutions to 
	\beq \label{eq:diff}
		k = a - a', \ a, a' \in S_3.
	\eeq
	For $k \neq 0$ there are only three possibilities.
	\\
	\textbf{Case 1.} $k$ has three one-coordinates and three two-coordinates. Then we can uniquely reconstruct $a$ and $a'$.
	\\
	\textbf{Case 2.} $k$ has four non-zero coordinates: two ones and two minus ones (note that in $\mathbb{F}_3$ we have $-1 = 2$). Then there are less than $n$ solutions to (\ref{eq:diff}). 
	\\
	\textbf{Case 3.} $k$ has only two non-zero coordinates: one of them is one and the other is equal to two. Then there are less than $n^2$ solutions to (\ref{eq:diff}). 
	\\

We will say that a pair of elements $(x, y)$ in $\mathbb{F}^n_3 \times \mathbb{F}^n_3$ (not necessarily in $S_3 \times S_3$) is a Case 1, 2, 3 pair if $x - y$ has respectively 6, 4, 2 non-zero coordinates, as described above. 

	%. Let $N(b)$ denote the set of neighbors of a vertex $b \in B$ and %$d(b)$ be the degree of $b$. We start with the following identity, where the sum runs over all ordered pairs of vertices (so $b$ might be equal to $b'$)
%	$$
%		\sum_b d^2(b) = \sum_{(b, b')} |N(b) \cap N(b')|.
%	$$	
	
	Let $H$ be a graph on $B$ such that for each $a \in S_3$ there is an edge between a single pair $(b, b')$ for which $a = b + b'$, similar to the one we have already encountered in the proof of Theorem \ref{thm:main}. Formally, we construct $H$ as a bipartite graph on two disjoint copies of $B$ (namely, $B_1$ and $B_2$) as color classes, and for each $a \in S_3$ we pick the smallest pair $(b_1, b_2)$ in the lexicographical order such that $a = b_1 + b_2$. We then place an edge between $b_m := \min(b_1, b_2) \in B_1$ and $b_M: = \max(b_1, b_2) \in B_2$, where again $\min$ and $\max$ are meant with respect to the lexicographical order. 
	
	The crucial observation is that if $v_1, v_2 \in B_1$ have a common neighbour, then the pair $(v_1, v_2)$ is either a Case 1, 2 or 3 pair. Indeed, then there is $w \in B_2$ such that both $v_1 + w \in S_3$ and $v_2 + w \in S_3$, so there are $a, a' \in S_3$ such that $v_1 - v_2 = a - a'$. 
	
	In fact, we are going to estimate 
	\beq \label{eq:CSSUM}
		 \sum_{(v_1, v_2) \in B_1 \times B_1} |N(v_1) \cap N(v_2)| = \sum_{w \in B_2} d^2(w).	
	\eeq
  For convenience, we will assume henceforth that vertices denoted by the letter $v$ belong to the first color class $B_1$ while those denoted by $w$ belong to the color class $B_2$.  The sum in (\ref{eq:CSSUM}) naturally splits into a sum over Case 1, Case 2 and Case 3 pairs, which we will estimate separately. Our aim is to bound each contribution from above by $O(n^4)$.
	
The contribution given by the Case 2 pairs is the hardest one. In fact the contribution can be as large as $n^5$, for example if $B_1 = B_2 = S_1 \cup S_2$. To overcome this difficulty we will excise some edges from $H$ while still retaining a positive proportion of $S_3$ represented by edges, but significantly reducing the contribution from Case 2 pairs.
	
	\vskip 1em
	\noindent \textbf{Case 1 pairs.}
	\vskip 1em	
	Observe that for any Case 1 pair $(v_1, v_2)$ one has $|N(v_1) \cap N(v_2)| \leq 1$, since for a common neighbour $w$ of $v_1, v_2$ one can uniquely reconstruct $v_1 + w$ and $v_2 + w$. Thus, we can simply bound the whole contribution from Case 1 pairs by $|B|^2$, which we may assume is at most $n^4$, otherwise we are done.

	\vskip 1em
	\noindent \textbf{Case 3 pairs.}
	\vskip 1em	
	
	Now we estimate the contribution of the pairs with $b - b' = (\ldots, 0, 1,  \ldots, 2, \ldots)$. We claim that for a fixed $k = (\ldots, 0, 1,  \ldots, 2, \ldots)$ one has
	\beq \label{eq:case3ineq}
	\sum_{b-b' = k} |N(b) \cap N(b')| \leq n^2.
	\eeq

In what follows we will use natural coordinate projections $\pi_I : \mathbb{F}_3^n \to \mathbb{F}_3^I$, where $I \subset [n]$. In particular, if $I$ is a singleton set, we write $\pi_i$ for the projection on the $i$th coordinate, and $\pi_{i, j}$ for the projection on the two-dimensional space spanned by the $i$th and $j$th coordinates. 
	
	Now let  $i_1, i_2$ be the indices such that $\pi_{i_1}(k) = 1$ and $\pi_{i_2}(k) = 2$, respectively. For an arbitrary common neighbor $b_c$  of $b$ and $b'$ one can write 
	$$
	(b + b_c) - (b' + b_c) = b - b' = k.
	$$
	Thus, $\pi_{i_1} (b + b_c) = \pi_{i_2}(b' + b_c) = 1$ and, moreover, there are exactly two indices $j_1, j_2$ such that
	$$
	\pi_{j_1, j_2} (b + b_c) = \pi_{j_1, j_2}(b' + b_c) = (1,1)
	$$ 	
	and the rest coordinates of $b+b_c$ and $b'+b_c$ are zero. We can thus map $b_c$ to a unique vector $v \in S_2$ with $\pi_{j_1, j_2}(v) = (1,1)$ (and the rest coordinates are zero). The set of common neighbours of Case-3 pairs $(b,b')$ with $b-b' = k$
	$$
	\bigcup_{b- b' = k} (N(b) \cap N(b'))
	$$ is then mapped the set of vectors $\{ v(k) \} \subset S_2 \subset \mathbb{F}_3^n$. The map is in fact injective, since we can reconstruct $b + b_c$ and $b' + b_c$ from $k$ and $v$. In turn, the pair of edges corresponding to $b + b_c$ and $b' + b_c$ uniquely identifies the pair $(b, b')$. Obviously the number of $v$'s is bounded by $n^2$, so we get (\ref{eq:case3ineq}). 
	
	Finally, there are at most $n^2$ choices of $k$ in (\ref{eq:case3ineq}), so the total contribution from Case 3 pairs to (\ref{eq:CSSUM}) is $O(n^4)$.

	\vskip 1em
	\noindent \textbf{Case 2 pairs.}
	\vskip 1em
	
   Let us get back to Case 2 pairs. The graph $H$ can be naturally decomposed into the union of graphs $H_i$, $i = 1 \dots n$, where $E(H_i) = \bar{S}_i := \{ v \in S_3 | \pi_i(v) = 1 \}$ and $V(H_i)$ is the set of vertices spanned by $E(H_i)$. In fact, $\bar{S}_i$ is a shifted $2$-sphere, so this decomposition corresponds to the decomposition of $S_3$ into a union of $n$ shifted two-spheres given by fixing the $i$th coordinate to one. 
   
  Let  $W_i \subset V(H_i)|_{B_2}$ be the set of vertices $w$ of $H_i$ in $B_2$ such that $d_{H_i}(w) > 2^{10}n$, where $d_{H_i}$ denotes degree in $H_i$. We claim that  
  \beq \label{eq:sublemmaclaim}
  \sum_{w \in W_i} d_{H_i}(w) \leq \frac{n^2}{50}.
  \eeq
  
  The claim follows from the sublemma below which essentially asserts that a large portion of the 2-sphere cannot be additively split into two unequal parts unless the size of one of the parts is comparable to the size of the whole sphere.
  
  \begin{lemma} \label{lm:sublemma}
  	Let $X, Y \subset \mathbb{F}^n_3$ with $|X| \leq n/2^{10}$ and $|Y| \leq n^2/100$. Then
  	$$
		| (X + Y) \cap S_2| \leq \frac{n^2}{50}.  	
  	$$
  \end{lemma}	  
  
  Before we proceed with the proof, let us show how the claimed bound (\ref{eq:sublemmaclaim}) follows from Lemma \ref{lm:sublemma}. Clearly, 
  $$
  |W_i| \leq \frac{|E(H_i)|}{2^{10}n} \leq \frac{n}{2^{10}}.
  $$ Let $\bar{S}'_i$ be the subset of $\bar{S}_i$ represented by all edges emanating from $W_i$. We have that $\bar{S}'_i \subset B_1 + W_i $. Equivalently, $\bar{S}'_i - 1_i \subset (B_1 - 1_i) + W_i$, where $1_i$ is the vector with $\pi_i(1_i) = 1$ and zero coordinates otherwise. 
  %We will now show that for two sets $W_1, W_2$ with $|W_1| \leq n/2^{10}$ and $|W_2| \leq n^2/100$ holds
  %$$
%	\left| (W_1 + W_2)  \cap S_2 \right| \leq n^2/50, 
  %$$
%and the claim will follow. 
Now either $|B_1| \geq n^2/100$ and we are done with the whole Lemma 4, or we apply Lemma \ref{lm:sublemma} with $X = W_i$ and $Y = B_1 - 1_i$, so that  
$$
 |\bar{S}'_i | \leq |(X+Y) \cap S_2| \leq  \frac{n^2}{50}.
$$ But $|\bar{S}'_i | = \sum_{w \in W_i} d_i(w)$ by construction.

\begin{remark}
	In fact, in the lines of Lemma \ref{lm:sublemma} one can prove the following more general result. Suppose $|A| \leq n/2^{10}$. Then 
	$$
		|(A+B) \cap S_2 | \leq \binom {n}{2}  - \binom {n - |A|}{2} + |B| \leq n|A| + |B|.
	$$
\end{remark}

\begin{proof}[Proof of Lemma \ref{lm:sublemma}]
	The intuition behind the proof is as follows. Let $L$ be the subspace spanned by $X$, so its dimension is at most $|X| \ll n$. Let $\pi$ be the projection to $L^{\perp}$ and assume that 
$|\pi(S_2)| \approx  |S_2|$ (we don't specify the exact meaning of $\approx$ here, but since $L$ is small, one may expect that $\pi(S_2)$ is comparable in size to the whole sphere). Then we can estimate 
$$
	|S_2 \setminus \left( (X + Y) \cap S_2 \right)| \geq |\pi(S_2 \setminus \left( (X + Y) \cap S_2 \right) )| \approx |S_2| - |Y|,
$$ 
since $|\pi(X +Y)| = |\pi(Y)| \leq |Y|$. Thus, $|(X + Y) \cap S_2| \approx |Y|$.

  Now we implement the outlined strategy. Let $L$ be subspace of $\mathbb{F}^n_3$ spanned by $X$, $l := \dim L$ and  $m := n - l$, so $m \geq (1 - 2^{-10})n$. Let $e_1, \ldots, e_m$ be standard basis elements, so $e_i = 1_i$ for $i = 1, \ldots, m$.  Without loss of generality we may assume that $\{ L, e_1, \ldots, e_m \}$ spans the whole space $\mathbb{F}^n_3$. Then let $T$ be a linear invertible transformation such that $T(e_i) = e_i$ for $1\leq i \leq m$ and $T(L) =\mathrm{Span}(\{e_{m+1}, \ldots, e_n \}) $. Denote by $M := \mathrm{Span}(\{e_{1}, \ldots, e_m \})$, by $M^c$ its set compliment and note that in particular $T(X) \perp M$.
  
 We have (as $T$ is one-to-one)
\ben 
	| (X+Y) \cap S_2| &=& |T(X + Y) \cap T(S_2)| \nonumber \\
							&=& |T(X + Y) \cap T(S_2) \cap M| + |T(X + Y) \cap T(S_2) \cap M^c|.  \label{eq:secondterms_}
\een

Since $T(X)$ has trivial orthogonal projection to $M$ and $T(X+Y) = T(X) + T(Y)$, we can crudely estimate
\ben 
	|T(X + Y) \cap T(S_2) \cap M| &\leq& |T(X+Y) \cap M| \nonumber \\
												&\leq& |\pi_M(T(X + Y))| = |\pi_M(T(Y))| \leq |Y|,  \label{eq:secondterms_1}
\een
where $\pi_M$ denotes orthogonal projection onto $M$. For the second term we have
\ben
 |T(X + Y) \cap T(S_2) \cap M^c| &\leq& |T(S_2) \cap M^c|  \label{eq:secondterms_2}  \\
 		&=& |T(S_2)| - |T(S_2) \cap M|. \nonumber
\een
But since $T(e_i + e_j) = T(e_i) + T(e_j) = e_i + e_j \in M$ for $1 \leq i, j \leq m$,  
\be
|T(S_2) \cap M| &\geq& \frac{m(m-1)}{2} \geq \frac{((1-2^{-10})n)((1-2^{-10})n - 1)}{2} \\
					    &\geq& 0.991\frac{n(n-1)}{2} \geq 0.99\frac{n^2}{2}, 
\ee
for $n$ large enough. Therefore,
\beq \label{eq:secondterms_3}
 |T(X + Y) \cap T(S_2) \cap M^c| \leq |T(S_2)| - 0.99\frac{n^2}{2} \leq 0.01\frac{n^2}{2}.
\eeq

Combining (\ref{eq:secondterms_}), (\ref{eq:secondterms_1}) and (\ref{eq:secondterms_3}) we obtain
$$
| (X+Y) \cap S_2| \leq 0.01\frac{n^2}{2} + \frac{n^2}{100} \leq \frac{n^2}{50},
$$
as required.

%Without loss of generality (reordering, if necessary, and excluding linearly dependent elements from $W_1$), there are $n - |W_1|$ basis elements $e_{|W_1| + 1}, \ldots, e_n$ ($e_i = 1_i$ is the vector %with the $i$th coordinate equal to one and the rest to zero), such that $\{W_1, e_{|W_1| + 1}, \ldots, e_n \}$ is linearly independent and span $\mathbb{F}^n_3$. So there is an invertible linear %transformation $T$ such that $\mathrm{Span}(T(W_1)) = \mathrm{Span}(\{e_1, \ldots, e_{|W_1|} \}) = \mathbb{F}^{|W_1|}_3$ and $T(e_i) = e_i$ for $i > |W_1|$. 
%Let $L$ be the subspace spanned by the last $n - |W_1|$ coordinates, so we have $\mathrm{Span}(T(W_i)) \cap L = 0$. Also, 
%$$
%T((W_1 + W_2) \cap S_2) = (T(W_1) + T(W_2)) \cap T(S_2)
%$$ 
%and
%\be
%	| (T(W_1) + T(W_2)) \cap T(S_2) \cap L| &=& | (T(W_1) \cap L + T(W_2) \cap L) \cap T(S_2)| \\
%																			& \leq & | T(W_2) \cap L| \leq |W_2| \leq \frac{n^2}{100}.
%\ee
%On the hand, $L$ is invariant under $T$, so $T(e_j + e_{j'}) = T(e_j) + T(e_{j'}) = e_j + e_{j'} \in S_2$ for $j, j' > |W_1|$. Thus
%$$
%|T(S_2) \cap L|  \geq \frac{((1 - 2^{-10})n)((1 - 2^{-10})n-1)}{2} \geq 0.991\frac{n(n-1)}{2} \geq 0.99\frac{n^2}{2}
%$$
%for $n$ large enough. Therefore, 
%$$
%\left| T((W_1 + W_2) \cap S_2) \right| \leq |S_2 \cap {L^{\perp}}| + |W_2| \leq (\frac{n^2}{2} - |S_2 \cap L|) + |W_2| \leq 0.01\frac{n^2}{2} + \frac{n^2}{100} \leq n^2/50.	
%$$  
\end{proof}
   
	We proceed with the proof of Lemma \ref{lm:additive_sphere}. 
 
Let us first recall some standard graph-theoretic notation. By $N_G(v)$ we denote the neighborhood of a vertex $v$ in a graph $G$. When the graph is omitted we assume that the neighborhood is taken in the ambient graph (in our case $H$). Similarly, by $d_G(v)$ we denote the degree of $v$ in $G$. 
 
  We prune each $H_i$, cutting all edges emanating from each $w \in W_i$, that is, from those vertices whose degree in $H_i$ is at least $2^{10}n$. By (\ref{eq:sublemmaclaim}), we will remove at most $n^2/50$ edges in total. Let us call the pruned graph $H'_i$ and count the number of common neighbours of all Case 2 pairs in $H'_i$. To emphasize the fact that we are counting common neighbours in $H'_i$ now, we denote by $N_{H'_i}(v)$ the set of neighbours of a vertex $v$ in $H'_i$. Note, however, that the property that a pair of vertices $(v_1, v_2)$ is a Case 2 pair is independent of the graph they belong to. 
  
We have  
  $$
	\sum_{(v_1, v_2) \text{ Case 2 in $H'_i$ }} |N_{H'_i}(v_1) \cap N_{H'_i}(v_2)| \leq \sum_{w \in V(H'_i)|_{B_2} } d_{H'_i}^2(w).  
  $$
On the other hand, by construction of $H'_i$, we have
\ben \label{eq:case2bound}
 \sum_{w \in V(H'_i)|_{B_2} } d_{H'_i}^2(w) &\leq& \max_ {w \in V(H'_i)|_{B_2}} d_{H'_i}(w) \left( \sum_{w \in V(H'_i)|_{B_2} } d_{H'_i}(w) \right) \\
 												   &\leq& 2^{10}n|\bar{S}_i| \ll n^3. \nonumber
\een

Let $H' := \bigcup_i H'_i$. By (\ref{eq:sublemmaclaim}), 
\beq \label{eq:edgenumber}
|E(H')| \geq |E(H)| - \frac{n^3}{50} \gg n^3.
\eeq Also, we claim that if $v_1w, v_2w \in E(H')$ and $(v_1, v_2)$ is a Case 2 pair in $H'$, then there is an $i$ such that $v_1w, v_2w \in E(H'_i)$, that is, $(v_1, v_2)$ is a Case 2 pair in $H'_i$. Indeed,  there is an $i$ such that $\pi_i(v_1 + w) = \pi_i(v_2 + w) = 1$, but then $v_1w, v_2w \in E(H'_i)$ by construction. Thus, by (\ref{eq:case2bound}),
\beq \label{eq:case2estimate}
	 \sum_{ (v_1, v_2) \textrm{  Case 2 in $H'$ }} |N_{H'}(v_1) \cap N_{H'}(v_2)| \leq \sum_{i = 1}^n \sum_{ (v_1, v_2) \textrm{  Case 2 in $H'_i$ }} |N_{H'_i}(v_1) \cap N_{H'_i}(v_2)| \ll n^4.
\eeq

Now for the graph $H'$ we combine the estimates for Case 1, Case 2 and Case 3 pairs (as $E(H') \subset E(H)$ we can safely use the upper bounds obtained for $H$ in Cases 1 and 3) and conclude
	\ben \label{eq:degree_square_estimate}
		\sum_{w \in B_2} d_{H'}^2(w) &=& \sum_{(v_1, v_2) \in B_1 \times B_1} |N_{H'}(v_1) \cap N_{H'}(v_2)| \\ \nonumber									   
			 &=& \sum_{v_1 - v_2 \textrm{  Case 1} } |N_{H'}(v_1) \cap N_{H'}(v_2)| \\ \nonumber
			 &+& \sum_{v_1 - v_2 \textrm{  Case 2 }} |N_{H'}(v_1) \cap N_{H'}(v_2)| \\ \nonumber
			 &+& \sum_{v_1 - v_2 \textrm{  Case 3 }} |N_{H'}(v_1) \cap N_{H'}(v_2)| \\ \nonumber
				& \ll &	 n^4.
	\een

	By the Cauchy-Schwarz inequality 
	$$
	\left(\sum_{w \in B_2}  d_{H'}(w) \right)^2	 \leq |B_2|\left( \sum_{w \in B_2} d^2_{H'}(w) \right),
	$$	
	so (\ref{eq:edgenumber}) and (\ref{eq:degree_square_estimate})  give
%	Finally, by Cauchy-Schwartz,
	$$
	n^6 \ll  |E(H')|^2 \ll \left(\sum_{w \in B_2} d_{H'}(w)\right)^2  \ll |B_2|n^4.
   $$
   Therefore, $|B_2| = |B| = \Omega(n^2)$.
	
\end{proof}

\section{Acknowledgments}
   The author thanks an anonymous referee for various suggestions which greatly improved the exposition. The author is also grateful to Andrew Granville for various refinements of the original argument, in particular for optimizing the lower bound, and for bringing the problem back to the author's attention. The author would also like to thank Peter Hegarty for very careful proofreading and Ilya Shkredov for discussions and coining the term ``discrete sphere''.

\end{document}